\documentclass[10pt]{amsart}

\usepackage{mdwlist}
\usepackage{mathpazo}
\usepackage[final]{microtype}
\usepackage{bm}
\usepackage{calc}

\usepackage[utf8]{inputenc}
\usepackage{textcomp}
\usepackage{xcolor}
\usepackage[pdftex]{graphicx}
\usepackage{xspace}
\usepackage{tikz}
\usetikzlibrary{arrows,shapes,backgrounds}

\usepackage{array}
\usepackage{multirow}

\long\def\commentout#1{}

\newif\ifprint
\printfalse
\ifprint
	\definecolor{linkred}{rgb}{0,0,0} % black
	\definecolor{linkblue}{rgb}{0,0,0} % black
\else
	\definecolor{linkred}{rgb}{0.7,0.2,0.2}
	\definecolor{linkblue}{rgb}{0,0.2,0.6}
\fi
\definecolor{green}{rgb}{0,0.8,0.1}

\PassOptionsToPackage{obeyspaces}{url} 
\usepackage[
	hypertexnames=true,% IM 2/10 prevent spurious duplicate dests
	hyperindex,
	pagebackref,
	pdftex,
	pdftitle={Total to parts},
	pdfdisplaydoctitle,
	pdfpagemode=UseNone,
	breaklinks=true,
	extension=pdf,
	bookmarks=false,
	plainpages=false,
	colorlinks,
	linkcolor=linkblue,
	citecolor=linkred,
	urlcolor=linkred,
	pdfmenubar=true,
	pdftoolbar=true,
	pdfpagelabels,
	pdfpagelayout=SinglePage,
	pdfview=Fit,
	pdfstartview=Fit
]{hyperref}%IM

\usepackage{titlesec}
\titleformat{\section}{\normalfont\large\bfseries}{\thesection.}{0.5em}{}[\kern0.em]
\titleformat{\subsection}[runin]{\normalfont\bfseries}{\thesubsection.}{0.2em}{}[\kern0.5em]
\titleformat{\subsubsection}[runin]{\normalfont\bfseries}{\thesubsubsection.}{0.1em}{}[\kern0.5em]

\makeatletter

\def\maketitle{\par
  \@topnum\z@ % this prevents figures from falling at the top of page 1
  \@setcopyright
  \thispagestyle{firstpage}% this sets first page specifications
  \ifx\@empty\shortauthors \let\shortauthors\shorttitle
  \else \andify\shortauthors
  \fi
  \@maketitle@hook
  \begingroup
  \@maketitle
  \toks@\@xp{\shortauthors}\@temptokena\@xp{\shorttitle}%
  \toks4{\def\\{ \ignorespaces}}% defend against questionable usage
  \edef\@tempa{%
    \@nx\markboth{\the\toks4
      \@nx{\the\toks@}}{\the\@temptokena}}%%IM Removed the conversion to upper case in the original line above. 
  \@tempa
  \endgroup
  \c@footnote\z@
  \@cleartopmattertags
}
\def\@settitle{\begin{center}%
  \baselineskip14\p@\relax
    \bfseries
\Large%IM INcrease size since no longer all caps
\@title%IM deleted uppercasenonmath
  \end{center}%
}
\def\author@andify{%
  \nxandlist {\unskip ,\penalty-1 \space\ignorespaces}%
    {\unskip {} \@@and~}%
    {\unskip ,\penalty-2 \space \@@and~}%
}
\def\@setauthors{%
  \begingroup
  \def\thanks{\protect\thanks@warning}%
  \trivlist
  \centering\footnotesize \@topsep30\p@\relax
  \advance\@topsep by -\baselineskip
  \item\relax
  \author@andify\authors
  \def\\{\protect\linebreak}%
{\large\authors}%IM INcrease size of author names
  \ifx\@empty\contribs
  \else
    ,\penalty-3 \space \@setcontribs
    \@closetoccontribs
  \fi
  \endtrivlist
  \endgroup
}
\def\@setaddresses{\par
  \nobreak \begingroup
\footnotesize
  \def\author##1{\nobreak\addvspace\bigskipamount}%
  \def\\{\unskip, \ignorespaces}%
  \interlinepenalty\@M
  \def\address##1##2{\begingroup
    \par\addvspace\bigskipamount\indent
    \@ifnotempty{##1}{(\ignorespaces##1\unskip) }%
    {\ignorespaces##2}\par\endgroup}%% IM removed \scshape
  \def\curraddr##1##2{\begingroup
    \@ifnotempty{##2}{\nobreak\indent\curraddrname
      \@ifnotempty{##1}{, \ignorespaces##1\unskip}\/:\space
      ##2\par}\endgroup}%
  \def\email##1##2{\begingroup
    \@ifnotempty{##2}{\nobreak\indent\emailaddrname
      \@ifnotempty{##1}{, \ignorespaces##1\unskip}\/:\space
      \ttfamily##2\par}\endgroup}%
  \def\urladdr##1##2{\begingroup
    \def~{\char`\~}%
    \@ifnotempty{##2}{\nobreak\indent\urladdrname
      \@ifnotempty{##1}{, \ignorespaces##1\unskip}\/:\space
      \ttfamily##2\par}\endgroup}%
  \addresses
  \endgroup
}
\renewenvironment{abstract}{%
  \ifx\maketitle\relax
    \ClassWarning{\@classname}{Abstract should precede
      \protect\maketitle\space in AMS document classes; reported}%
  \fi
  \global\setbox\abstractbox=\vtop \bgroup
    \normalfont\Small
    \list{}{\labelwidth\z@
      \leftmargin3pc \rightmargin\leftmargin
      \listparindent\normalparindent \itemindent\z@
      \parsep\z@ \@plus\p@
      
    }%
    \item[\hskip\labelsep\bfseries\abstractname.]%
}{%
  \endlist\egroup
  \ifx\@setabstract\relax \@setabstracta \fi
}

\def\section{\@startsection{section}{1}%
  \z@{.7\linespacing\@plus\linespacing}{.5\linespacing}%
  {\normalfont\bfseries\centering}}% IM replaced \scshape with bfseries

\renewcommand*\l@section{\@tocline{1}{0pt}{0em}{1.75em}{}}
\renewcommand*\l@subsection{\@tocline{2}{0pt}{1.75em}{2em}{}}

\def\ps@normal{\def\@oddhead{\small \hfill \leftmark \hfill\thepage }
\def\@evenhead{\small \thepage \hfill \rightmark \hfill}
\def\@oddfoot{}
\def\@evenfoot{}}
\commentout{
}
\makeatother

\usepackage[backrefs,msc-links,nobysame]{amsrefs}

\renewcommand{\eprint}[1]{#1}

\makeatletter
\def\bib@div@mark#1{%
 \@mkboth{{#1}}{{#1}}%
	}
\def\print@backrefs#1{%
 \space\SentenceSpace$\leftarrow$\csname br@#1\endcsname
}
\catcode`\'=11 % letter
\renewcommand{\PrintAuthors}[1]{%
 \ifx\previous@primary\current@primary
  \sameauthors\@empty
 \else
  \def\current@bibfield{\bib'author}%
		  \PrintNames{}{}{\scshape #1}%
 \fi
}
\catcode`\'=12 % other
\def\MRhref#1#2{%
 \begingroup
 \parse@MR#1 ()\@empty\@nil%
  \href{\MR@url}{\texttt{\@tempd\vphantom{()}}}%
  \ifx\@tempe\@empty
  \else
   \ \href{\MR@url}{\texttt{(\@tempe)}}%
  \fi
 \endgroup
}%
\def\MR#1{%
 \relax\ifhmode\unskip\spacefactor3000 \space\fi
 \begingroup
 \strip@MRprefix#1\@nil
  \edef\@tempa{\@nx\MRhref{MR\@tempa}{\@tempa}}%
 \@xp\endgroup
 \@tempa
}
\makeatother

\newcommand{\loosespread}{\linespread{1.15}\normalfont\selectfont}

\newcommand{\bibspread}{\linespread{1.02}\normalfont\selectfont}
\renewcommand{\arraystretch}{1.2}

\newcommand{\noparskip}{\parskip0pt}
\newcommand{\smallparskip}{\parskip3pt}

\makeatletter
\newtheoremstyle{questions}{\thm@preskip}{\thm@preskip}%
     {\itshape}%         Body font
     {}%         Indent amount (empty = no indent, \parindent = para indent)
     {\bfseries}% Thm head font
     {}%        Punctuation after thm head
     {6pt}%     Space after thm head (\newline = linebreak)
     {\thmname{#1}\thmnumber{ #2}\thmnote{ \textbf{~[#3]}}}%         Thm head spec
\makeatother

\theoremstyle{plain}
\newtheorem{thm}[equation]{Theorem}
\newtheorem{prop}[equation]{Proposition}
\newtheorem{lem}[equation]{Lemma}
\newtheorem{cor}[equation]{Corollary}

\newtheorem{conjprob}[equation]{Conjecture-Problem}

\theoremstyle{questions}% see preamble
\newtheorem{ttsque}[equation]{Total-to-parts Question}

\theoremstyle{definition}
\newtheorem{rem}[equation]{Remark}
\newtheorem{ex}[equation]{Example}
\newtheorem{exs}[equation]{Examples}

\newcommand{\ds}[1]{\displaystyle{#1}}
\newcommand{\isom}{\cong}
\newcommand{\sspan}{\operatorname{span}}
\newcommand{\Aut}[1]{\operatorname{Aut}}
\newcommand{\RR}{\mathbb{R}}
\newcommand{\sq}[1]{[#1]}
\newcommand{\ang}[1]{\langle{#1}\rangle}
\newcommand{\kk}{\mathbf{k}}
\newcommand{\dd}{\mathbf{d}}
\newcommand{\rr}{\mathbf{r}}
\newcommand{\ddhat}{\widehat{\mathbf{d}}}
\newcommand{\chat}{\widehat{c}}
\newcommand{\pp}{\mathbf{p}}

\newcommand{\ii}{\mathbf{i}}
\newcommand{\ff}{\mathbf{f}}
\newcommand{\FF}{\mathbf{F}}
\newcommand{\ee}{\mathbf{e}}
\renewcommand{\SS}{\mathbf{S}}
\newcommand{\CCC}{\mathbb{C}}
\newcommand{\PPP}{\mathbb{P}}
\newcommand{\AAA}{\mathbb{A}}
\newcommand{\VVV}{\mathbb{V}}
\newcommand{\QQ}{\mathbb{Q}}
\newcommand{\thst}[2]{\ensuremath{{#1}^{\mathrm{#2}}}}

\newcommand{\subsecext}[1]{{{\small\S\S}\kern1.5pt{#1}}}
\newcommand{\subsec}[1]{{{\small\S\S}\kern1.5pt\ref{#1}}}
\newcommand{\sect}[1]{{section~\ref{#1}}}
\newcommand{\sects}[2]{{sections~\ref{#1} and~\ref{#2}}}

\newcommand{\Sects}[2]{{Sections~\ref{#1} and~\ref{#2}}}
\newcommand{\question}[1]{Question~\ref{#1}}

\renewcommand{\theequation}{\thesubsection.\arabic{equation}}
\numberwithin{equation}{subsection}

\def\calcbaseurl{http://faculty.fordham.edu/dswinarski/totaltoparts/}
\newcommand{\calcpage}[1]{\href{\calcbaseurl#1}{{\sffamily{\texttt{#1}}}}}
\newcommand{\calcpagecite}[1]{\cite{codesamples}*{\href{\calcbaseurl#1}{{\sffamily{\texttt{#1}}}}}}
\newcommand{\calctoccite}[1]{\edef\ctanch{\calcbaseurl toc.htm\##1}\cite{codesamples}*{\href{\ctanch}{{\sffamily{\texttt{toc.htm\##1}}}}}}

\newcommand{\neturl}[1]{\href{#1}{{\sffamily{\texttt{#1}}}}}
\newcommand{\neturltilde}[2]{\href{#1}{{\sffamily{\texttt{#2}}}}}

\newif\ifauthornotes\authornotesfalse

\long\def\commentout#1{}% for hiding blocks of text

\newif\ifshrink % true when we are making the short version for submission
\newif\ifextracted % true when we are making an extracted support file for the the short version

\shrinkfalse\extractedfalse

\begin{document}
\loosespread
\noparskip 

\author{Ian Morrison}
\address{Department of Mathematics\\Fordham University\\Bronx, NY 10458} 
\email{morrison@fordham.edu}

\author{David Swinarski}					
\address{Department of Mathematics\\Fordham University\\Bronx, NY 10458} 
\email{dswinarski@fordham.edu}

\title{Can you play a fair game of craps with a loaded pair of dice?}


\begin{abstract}
We study, in various special cases, total distributions on the product of a finite collection of finite probability spaces and, in particular, the question of when the probability distribution of each factor space is determined by the total distribution.\ifshrink\vskip-48pt\hbox{~}\fi
%\begin{center}\textbf{\normalsize{}Authors' draft of \today. Please do not circulate.\kern4pt}\end{center}
\end{abstract}

%\par
%\leavevmode\hbox{}
%\vskip-12pt

\maketitle
\thispagestyle{empty}% suppress page number on title pagesr

\par
\leavevmode\hbox{}
\ifshrink~$~$\vskip-96pt\else\vskip-24pt\tableofcontents\smallparskip\fi

\section{Introduction}\label{intro}\stepcounter{subsection}

Craps is played by rolling two standard cubical dice with sides numbered $1$ to~$6$.
\ifshrink
Play is determined by the totals that show whose probabilities we call the \emph{total distribution} of the dice. For details of the rules and the pretty probability arguments computing the odds of winning, see \calcpagecite{1.craps}. 
\else
For readers who may not have seen them, we review in an Appendix (\sect{craps}) the rules of the game and the pretty probability arguments needed to compute the odds of winning. Here we simply note that the basic play (ignoring common side bets) depends only on the total of the numbers showing on the two dice. We call the probabilities of these totals the \emph{total distribution} of the two dice.
\fi
The dice are called \emph{fair} if each number is equally likely to be rolled and \emph{loaded} otherwise.
The question in the title asks whether there are any pairs of loaded dice for which this total distribution matches that of a fair pair. With such a loaded pair of dice, it would be possible to play a game of craps with the same probabilities of game events as if fair dice were being used. We give three proofs (in Proposition~\ref{crapsfairno},
\ifshrink
\calcpagecite{2.solutions}%
\else
Table~\ref{twodicefiftyone}%
\fi%
, and Corollary~\ref{crapsfairnobis}) that no such loaded dice exist. The first two proofs involve solving first partially then fully an explicit, dependent set of $11$ quadratic and $2$ linear equations in the $12$ side probabilities using a lengthy computer calculation by the mathematical software package \texttt{Magma}. The third follows by hand in a few lines from a very different formulation of the problem.

We were led to this reformulation by studying a much larger class of questions of this type that are the real focus of the paper. We consider dice of any order $k\ge 2$ whose sides may have arbitrary probabilities, and call any finite set of such dice, of possibly different orders, a \emph{sack} $\SS$. The probabilities of seeing each possible total when the dice in a sack are rolled we call its total distribution, and we ask what we call the \emph{total-to-parts question}: When are the probability distributions of all the dice in the sack determined by its total distribution?  In some cases, which we call \emph{exotic}, the answer to this more general question turns out to be negative; for example, computational evidence suggests that there are exotic pairs of dice of any order greater than or equal to $12$ (cf.~Conjecture-Problem~\ref{exoticpairsexist}). In fact, we generalize still further by allowing arbitrary complex numbers as probabilities.  This is explained in \sect{totalpartssetup}, where all the terms used in this introduction are defined. Our answers are given in \sects{multpoly}{exoticsacks}.

We have taken the unusual step of leaving some of the scaffolding of our work on these questions standing in this paper. Our motivation is to provide an example to novice mathematicians of the embryology of a research project. In one direction, we hope to illustrate both how generalizing a problem often, paradoxically, places it in a simpler context. In the opposite direction, the more general context may reveal interesting new special cases whose analysis offers new avenues of attack on the original problem. 

Here, very briefly is how our work illustrates these ideas. After setting up the largest class of questions that we deal with in \subsec{defnot} and \subsec{totalparts}, we explain in \subsec{geometric} how we can fit \emph{all} the part and total distributions for a fixed combinatorial type or vector of orders $\kk$ into a function $\FF_{\kk}: \CCC^T \to \CCC^T$, each of whose $T$ coordinate functions is polynomial in the $T$ input variables $\null$\footnote{We use bracketed text like this to provide translations into the language of algebraic geometry, in which both authors work, for readers familiar with it, but other readers can, with no loss, simply ignore these insertions.}[or an affine morphism. In other words, the generalizations fit into algebraic families.] 

Each \emph{individual} generalization (for example, the title question itself) concerns a fiber of this function---the inverse image of a \emph{single} point in its range. The answer to such questions can often be deduced for most or even all points [or generically] from geometric invariants of $\FF_{\kk}$ that can be much simpler to calculate. Further, these geometric invariants can often be calculated by understanding any suitably ``nice'' fiber [or specializing]. We may be able to find ``nice'' fibers that are particularly easy to analyze, and then apply our analysis to a fiber of particular interest, like the craps fiber [or generalize].  This is how we answer a number of questions stated in \subsec{totalparts} in \sect{multpoly}.

Another motivation for recapitulating our progress on the total-to-parts question is to highlight the way apparently irrelevant areas of mathematics have of unlocking problems, and the corollary value of having the broadest possible exposure to all areas of mathematics. This is already clear from the preceding paragraphs. The title question concerns probability distributions, but by having some familiarity with algebraic geometry, we make it easier. The sequel will make our point even clearer. The insight of Mike Stillman, discussed below, provides a new way to think about the maps $\FF_{\kk}$ that greatly simplifies all  questions in \sect{multpoly}. It also leads to constructions of exotic pairs of dice in \sect{exoticsacks} that rely on understanding factorizations of cyclotomic polynomials. We hope that an exposition organized to make these precepts stand out will justify the overhead incurred.

The work discussed here developed as follows. Our interest was provoked by a talk given at Fordham by Jordan Stoyanov in which he asked whether a pair of non-standard ``dice'' might have the same total distribution as a standard pair. The quotation marks here are because he defined dice differently than we do: Stoyanov allowed the number of spots on the sides of his ``dice'' to be non-standard, but required his ``dice'' to have a uniform probability distribution. With the requirement that the number of spots be positive integers, there is, up to permutation, a unique non-standard pair of six sided dice with sides marked $(1,2,2,3,3,4)$ and $(1,3,4,5,6,8)$ (cf.~\cite{Stoyanov}*{\subsecext{12.7}}) as the reader may find it an amusing exercise to verify.  

We both felt immediately that the total-to-parts question considered here was a more natural one. Indeed, during Stoyanov's talk, we worked out the case of two coins (given in Example~\ref{twocoinsfair}, but an easy exercise that we urge the reader to attempt before going further) and wrote down the equations which encode solutions for a pair of standard dice (analyzed in Example~\ref{twosixsideddice}). We then passed through a series of less special cases uncovering many pretty arguments along the way, some of which we give here. These included cases where the total distribution was that of a pair of fair dice of small orders and of a sack of fair coins (reviewed in \sect{coins}). It was this last case that first suggested to us---tardily, the reader may justly remark---considering the general total-to-parts problem. Trying to generalize our arguments and failing, we retreated to the case of a coin and a die whose analysis, reviewed in \sect{coindie}, hinted at a connection to factorizations of polynomials. 

At this point, we mentioned the problem and our results to Dave Bayer, who was on sabbatical at MSRI. Dave passed it on to Mike Stillman (also visiting MSRI) over a BBQ dinner, and Mike immediately saw that there was a completely different metaphor for our general problem. When Dave emailed me Mike's idea, the first word in my reply was ``Wow!'' because it was immediately clear that this new perspective gave easy answers to many of the questions we had not yet settled, and shorter proofs of many of our results. Readers are invited to follow our progress up to the end of \sect{coindie}, and then, we hope, will share the delight we felt on seeing this work radically simplified by Mike's insight when we reveal it in \sect{multpoly}. \Sects{multpoly}{exoticsacks} use this insight to produce many exotic sacks of two dice and poses some intriguing asymptotic conjectures about them that we leave as open problems for interested readers.  

Computer calculations in played an essential role in our investigation.  We used the commercial software packages \texttt{Magma} and \texttt{Maple}\texttrademark\ \cites{Magma, Maple} as well as the free, open-source packages \texttt{Bertini}, \texttt{Macaulay2}, \texttt{QEPCAD-B}, and \texttt{Sage} \cites{Bertini, Macaulay, QEPCADB, SAGE}.  We have posted many of our calculations online~\cite{codesamples}, typically giving both the input commands and output of an interactive session.  We have tried to make these sessions informal introductions to the packages by commenting each step and hope that readers will find models in our code for calculations they want to make.  

\subsection*{Acknowledgements} We thank Jordan Stoyanov for suggesting the problem. 
  We also offer special thanks to Mike Stillman for the insight underpinning all of \sect{multpoly} and to Dave Bayer for eliciting and communicating it.

\section {First examples}\label{crapsfairsection}\stepcounter{subsection}

During Stoyanov's lecture we already worked out the simplest example, replacing the pair of six-sided dice by two coins.  

\begin{ex}\label{twocoinsfair}
In tossing two fair coins, the probabilities of seeing $0$, $1$ or $2$ heads---i.e., $(HH, HT \text{~or~} TH, TT)$---are $(\frac{1}{4},\frac{1}{2},\frac{1}{4})$.
Let $p$ and $q$ be the probabilities of heads on the two possibly loaded coins, so that the probabilities of tails are $(1-p)$ and $(1-q)$.  Then to obtain the total distribution of two fair coins we must have:
\begin{displaymath}
\begin{array}{rcl}
pq & = & \frac{1}{4}\\
p(1-q)+(1-p)q & = &  \frac{1}{2}\\
(1-p)(1-q) & = &  \frac{1}{4}\\
\end{array}
\end{displaymath}
Substituting the first equation into the third gives $p+q=1$.  Substituting $q=1-p$ into the first equation yields $p^2-p+\frac{1}{4}=(p-\frac{1}{2})^2= 0$  so $p = q = \frac{1}{2}$. 
\end{ex}

The process of substitution used above is called \emph{elimination} and is, no doubt, familiar from solving systems of linear equations. During Stoyanov's lecture, we realized that the case of two dice could be answered by these methods and wrote down the equations below. We just as quickly realized that the complexity of the calculations called for the help of one of the many computer algebra systems that implement algorithms, analogous to, but more intricate than, that of Gauss-Jordan, for elimination in systems of \emph{polynomial} equations. A few cases simple enough that elimination can be carried out by hand are treated in \sect{coindie}.

\begin{ex} \label{twosixsideddice}
Let $p_i$ and $q_j$ be the probabilities that the first and second dice come up showing each number from $1$ to $6$. Now we have $36$ possible rolls $(i,j)$, each of which would have probability $\frac{1}{6}\cdot \frac{1}{6}= \frac{1}{36}$ if the dice were fair. Grouping these rolls first by the total $i+j$ and then by the number of rolls yielding each total gives the following system of equations.

\noindent%\begin{center}%
{\setlength{\tabcolsep}{0.1em}%
\begin{tabular}{>{$}r<{$}>{$}l<{$}>{$}c<{$}>{$}r<{$}>{$}l<{$}}%
	p_1 q_1 &=& \frac{1}{36} &=&p_6 q_6\\
	p_1 q_2 + p_2 q_1 &=& \frac{2}{36} &=& p_5 q_6 + p_6 q_5 \\
	p_1 q_3 + p_2 q_2 + p_3 q_1 &=& \frac{3}{36} &=&p_4 q_6 + p_5 q_5 + p_6 q_4 \\
\ifextracted\else\refstepcounter{equation}\label{twodiceequations}\text{\lower1pt\hbox{\textbf{(\theequation)}\kern38pt}}\fi
	p_1 q_4 + p_2 q_3 + p_3 q_2 + p_4 q_1  &=& \frac{4}{36} &= &p_3 q_6 + p_4 q_5 + p_5 q_4 + p_6 q_3 \\
	p_1 q_5 + p_2 q_4 + p_3 q_3 + p_4 q_2 + p_5 q_1   &= &\frac{5}{36} &=&p_2 q_6 + p_3 q_5 + p_4 q_4 + p_5 q_3 + p_6 q_2 \\
	p_1 q_6 + p_2 q_5 + p_3 q_4 + p_4 q_3 + p_5 q_2 + p_6 q_1   &=& \frac{6}{36} &&\\
	\displaystyle{\sum_{i=1}^{6}} p_i &=& 1 &= & \displaystyle{\sum_{i=1}^{6}} q_i\,.
\end{tabular}
}
%\end{center}

Moreover, despite the fact that we have $13$ equations in only $12$ unknowns, we expected to find only a \emph{finite} number of solutions, because these equations are dependent: the sum of all the products on the first six lines is the product of the sums on the last. As an easy exercise, we invite the reader to write down the analogous system of $7$ equations in $6$ unknowns for triangular dice.
\end{ex}

Algebraic geometry is the study of systems of polynomial equations.  However, in addition to the $13$ equations above, we also want the $p_i$ and $q_j$ to be real and to be between $0$ and $1$.  Thus, we have a system of polynomial equations and \emph{inequalities}. Such systems are the subject of \emph{semialgebraic geometry}, and our first approach was to apply its techniques.  

The first package we used to investigate whether fair dice give the only solution was \texttt{QEPCAD-B}~\cite{QEPCADB}. The acronym stands for \textbf{Q}uantifier \textbf{E}limination by \textbf{P}artial \textbf{C}ylindrical \textbf{A}lgebraic \textbf{D}ecomposition, and the tool combines polynomial manipulations with boolean algebra to determine whether or not positive real valued solutions exist without actually finding them. The package confirmed two coins must be fair, but crashed when we tried the analogous calculations for dice with more sides~\calcpagecite{2.qepcadb.htm}.

After this brief foray into semialgebraic geometry, we returned to our roots as algebraic geometers and analyzed the system (\ref{twodiceequations}) by elimination.  We turned to \texttt{Magma}~\cite{Magma}, a sophisticated commercial computer algebra system, which performs elimination via Gr\"obner bases in a way that more directly generalizes Gaussian elimination. We will not explain Gr\"obner bases and their applications to elimination here, but we warmly recommend \cite{CoxLittleOShea}*{Chapters 2 and 3} for an introduction. We first used \texttt{Magma} to eliminate all the variables except $q_6$~\calcpagecite{2.magma-a.htm}, very much as in the example of two coins above, providing a single polynomial for $q_6$ whose roots are the possible values of $q_6$ at different solution points.  In particular, only finitely many values of $q_6$ occur in the solutions of the system.  We then substituted each possible positive real value for $q_6$ into the original equations and eliminated all the variables except $q_5$ to find the possible pairs $(q_5,q_6)$ with both values real and positive.   By repeating this process several times, we eventually found that the only solution of these thirteen equations with all coordinates real and positive is a pair of fair dice. This constituted our first proof of the following fact:

\begin{prop}
\label{crapsfairno} You can't play a fair game of craps with a loaded pair of dice.
\end{prop}

\ifshrink 
We then used \texttt{Magma}~\calcpagecite{2.magma-b.htm} to describe the solution set [affine scheme] of (\ref{twodiceequations}) geometrically, first finding that it is finite and that if the solutions are counted with multiplicities, there are $252$ [the solutions form zero-dimensional scheme of degree 252]. We then divided the solutions into groups definable by equations with rational coefficients [found the irreducible components over $\QQ$ of our scheme] and noticed that all of the complex points had coordinates in the field $\QQ[\sqrt{-3}]$.  This allowed us to compute all $51$ complex solutions of the system exactly. 
They are shown at~\calcpagecite{2.solutions}, grouped into $25$ pairs where the dice differ and the fair solution where the two dice are identical.  Of the nonfair solutions, only one is real, and it contains negative entries. This piqued our interest in what happens for dice with any number of sides, and suggested that the natural context for answering these questions needed to allow arbitrary complex ``probabilities''.
\else
Although we were pleased to have an answer to the title question of the paper, we wanted to understand the system (\ref{twodiceequations}) better.  We used  \texttt{Magma} again in~\calcpagecite{2.magma-b.htm} to compute the dimension and degree of the solution set [affine scheme] of this system of equations, and discovered that it is zero-dimensional of degree 252.  This means that the system has 252 complex solutions (counted with multiplicity).  We decomposed this scheme into its irreducible components over $\QQ$---that is, irrational solutions presented in groups definable by equations with rational coefficients--- and noticed that all of the complex points had coordinates in the subfield $\QQ[\sqrt{-3}] \subset \mathbb{C}$.  This allowed us to compute all $51$ complex solutions of the system exactly. 
The solutions are shown in Table~\ref{twodicefiftyone}, where they are grouped into $25$ pairs where the dice differ, and the fair solution where the two dice are identical.  Of the nonfair solutions, only the pair in the second line is real, and it contains negative entries.  Also, in the table, we write the solutions in terms of the number $\zeta = \zeta_6 = e^\frac{2\pi i}{6}= \frac{1}{2}(1+\sqrt{3}\,i)$ instead of $\sqrt{-3}$; the reason we prefer this notation will become clear in \sect{multpoly}.

\begin{table}[th!]
	\begin{center}
\refstepcounter{equation}\label{twodicefiftyone} \textbf{Table~\ref{twodicefiftyone}} \quad Solutions in terms of $\zeta = \frac{1+\sqrt{3}\,i}{2} = e^\frac{2\pi i}{6}$ of the system ~\ref{twodiceequations} 
{\renewcommand{\arraystretch}{1.3}\setlength{\tabcolsep}{0.2em}\small \begin{tabular}{>{$}c<{$}  >{$}c<{$} >{$}c<{$} >{$}c<{$} >{$}c<{$} >{$}c<{$} c >{$}c<{$} >{$}c<{$} >{$}c<{$} >{$}c<{$} >{$}c<{$} >{$}c<{$}}
p_1&  p_2&  p_3&  p_4&  p_5&  p_6&  & q_1&  q_2&  q_3&  q_4&  q_5&  q_6  \\
\frac{1}{6}&\frac{1}{6}&\frac{1}{6}&\frac{1}{6}&\frac{1}{6}&\frac{1}{6}&&\frac{1}{6}&\frac{1}{6}&\frac{1}{6}&\frac{1}{6}&\frac{1}{6}&\frac{1}{6}\\
\frac{1}{2}&\frac{-1}{2}&\frac{1}{2}&\frac{1}{2}&\frac{-1}{2}&\frac{1}{2}&&\frac{1}{18}&\frac{1}{6}&\frac{5}{18}&\frac{5}{18}&\frac{1}{6}&\frac{1}{18}\\
\frac{-1}{6}&\frac{-4\zeta+1}{6}&\frac{-4\zeta+7}{6}&\frac{4\zeta+3}{6}&\frac{4\zeta-3}{6}&\frac{-1}{6}&&\frac{-1}{6}&\frac{4\zeta-3}{6}&\frac{4\zeta+3}{6}&\frac{-4\zeta+7}{6}&\frac{-4\zeta+1}{6}&\frac{-1}{6}\\
\frac{-\zeta-1}{9}&\frac{-5\zeta+4}{9}&\frac{-\zeta+5}{9}&\frac{\zeta+4}{9}&\frac{5\zeta-1}{9}&\frac{\zeta-2}{9}&&\frac{\zeta-2}{12}&\frac{2\zeta-1}{4}&\frac{5\zeta+5}{12}&\frac{-5\zeta+10}{12}&\frac{-2\zeta+1}{4}&\frac{-\zeta-1}{12}\\
\frac{-\zeta}{3}&\frac{-\zeta+3}{3}&\frac{3\zeta-2}{3}&\frac{-3\zeta+1}{3}&\frac{\zeta+2}{3}&\frac{\zeta-1}{3}&&\frac{\zeta-1}{12}&\frac{4\zeta-1}{12}&\frac{3\zeta+4}{12}&\frac{-3\zeta+7}{12}&\frac{-4\zeta+3}{12}&\frac{-\zeta}{12}\\
\frac{-\zeta-1}{12}&\frac{-2\zeta+1}{4}&\frac{-5\zeta+10}{12}&\frac{5\zeta+5}{12}&\frac{2\zeta-1}{4}&\frac{\zeta-2}{12}&&\frac{\zeta-2}{9}&\frac{5\zeta-1}{9}&\frac{\zeta+4}{9}&\frac{-\zeta+5}{9}&\frac{-5\zeta+4}{9}&\frac{-\zeta-1}{9}\\
\frac{-\zeta}{6}&\frac{-\zeta+1}{3}&\frac{-\zeta+3}{6}&\frac{\zeta+2}{6}&\frac{\zeta}{3}&\frac{\zeta-1}{6}&&\frac{\zeta-1}{6}&\frac{\zeta}{3}&\frac{\zeta+2}{6}&\frac{-\zeta+3}{6}&\frac{-\zeta+1}{3}&\frac{-\zeta}{6}\\
\frac{-2\zeta+1}{6}&\frac{1}{2}&\frac{2\zeta-1}{6}&\frac{-2\zeta+1}{6}&\frac{1}{2}&\frac{2\zeta-1}{6}&&\frac{2\zeta-1}{18}&\frac{4\zeta+1}{18}&\frac{2\zeta+5}{18}&\frac{-2\zeta+7}{18}&\frac{-4\zeta+5}{18}&\frac{-2\zeta+1}{18}\\
\frac{-2\zeta+1}{9}&\frac{-\zeta+2}{9}&\frac{-2\zeta+4}{9}&\frac{2\zeta+2}{9}&\frac{\zeta+1}{9}&\frac{2\zeta-1}{9}&&\frac{2\zeta-1}{12}&\frac{\zeta}{4}&\frac{\zeta+4}{12}&\frac{-\zeta+5}{12}&\frac{-\zeta+1}{4}&\frac{-2\zeta+1}{12}\\
\frac{-\zeta+1}{3}&\frac{\zeta}{3}&\frac{-\zeta+1}{3}&\frac{\zeta}{3}&\frac{-\zeta+1}{3}&\frac{\zeta}{3}&&\frac{\zeta}{12}&\frac{2\zeta+1}{12}&\frac{\zeta+3}{12}&\frac{-\zeta+4}{12}&\frac{-2\zeta+3}{12}&\frac{-\zeta+1}{12}\\
\frac{-\zeta+2}{3}&\frac{\zeta-1}{1}&\frac{-5\zeta+4}{3}&\frac{5\zeta-1}{3}&\frac{-\zeta}{1}&\frac{\zeta+1}{3}&&\frac{\zeta+1}{36}&\frac{2\zeta+5}{36}&\frac{\zeta+10}{36}&\frac{-\zeta+11}{36}&\frac{-2\zeta+7}{36}&\frac{-\zeta+2}{36}\\
\frac{-2\zeta+1}{12}&\frac{-\zeta+1}{4}&\frac{-\zeta+5}{12}&\frac{\zeta+4}{12}&\frac{\zeta}{4}&\frac{2\zeta-1}{12}&&\frac{2\zeta-1}{9}&\frac{\zeta+1}{9}&\frac{2\zeta+2}{9}&\frac{-2\zeta+4}{9}&\frac{-\zeta+2}{9}&\frac{-2\zeta+1}{9}\\
\frac{-\zeta+1}{4}&\frac{1}{4}&\frac{\zeta}{4}&\frac{-\zeta+1}{4}&\frac{1}{4}&\frac{\zeta}{4}&&\frac{\zeta}{9}&\frac{\zeta+1}{9}&\frac{\zeta+2}{9}&\frac{-\zeta+3}{9}&\frac{-\zeta+2}{9}&\frac{-\zeta+1}{9}\\
\frac{-\zeta+1}{6}&\frac{-\zeta+1}{6}&\frac{1}{3}&\frac{1}{3}&\frac{\zeta}{6}&\frac{\zeta}{6}&&\frac{\zeta}{6}&\frac{\zeta}{6}&\frac{1}{3}&\frac{1}{3}&\frac{-\zeta+1}{6}&\frac{-\zeta+1}{6}\\
\frac{-\zeta+2}{6}&\frac{0}{1}&\frac{\zeta+1}{6}&\frac{-\zeta+2}{6}&\frac{0}{1}&\frac{\zeta+1}{6}&&\frac{\zeta+1}{18}&\frac{\zeta+1}{9}&\frac{\zeta+4}{18}&\frac{-\zeta+5}{18}&\frac{-\zeta+2}{9}&\frac{-\zeta+2}{18}\\
\frac{1}{3}&\frac{-\zeta}{3}&\frac{2\zeta}{3}&\frac{-2\zeta+2}{3}&\frac{\zeta-1}{3}&\frac{1}{3}&&\frac{1}{12}&\frac{\zeta+2}{12}&\frac{\zeta+2}{12}&\frac{-\zeta+3}{12}&\frac{-\zeta+3}{12}&\frac{1}{12}\\
\frac{\zeta+1}{3}&\frac{-\zeta}{1}&\frac{5\zeta-1}{3}&\frac{-5\zeta+4}{3}&\frac{\zeta-1}{1}&\frac{-\zeta+2}{3}&&\frac{-\zeta+2}{36}&\frac{-2\zeta+7}{36}&\frac{-\zeta+11}{36}&\frac{\zeta+10}{36}&\frac{2\zeta+5}{36}&\frac{\zeta+1}{36}\\
\frac{-\zeta}{12}&\frac{-4\zeta+3}{12}&\frac{-3\zeta+7}{12}&\frac{3\zeta+4}{12}&\frac{4\zeta-1}{12}&\frac{\zeta-1}{12}&&\frac{\zeta-1}{3}&\frac{\zeta+2}{3}&\frac{-3\zeta+1}{3}&\frac{3\zeta-2}{3}&\frac{-\zeta+3}{3}&\frac{-\zeta}{3}\\
\frac{-2\zeta+1}{18}&\frac{-4\zeta+5}{18}&\frac{-2\zeta+7}{18}&\frac{2\zeta+5}{18}&\frac{4\zeta+1}{18}&\frac{2\zeta-1}{18}&&\frac{2\zeta-1}{6}&\frac{1}{2}&\frac{-2\zeta+1}{6}&\frac{2\zeta-1}{6}&\frac{1}{2}&\frac{-2\zeta+1}{6}\\
\frac{-\zeta+1}{6}&\frac{1}{3}&\frac{\zeta}{6}&\frac{-\zeta+1}{6}&\frac{1}{3}&\frac{\zeta}{6}&&\frac{\zeta}{6}&\frac{1}{3}&\frac{-\zeta+1}{6}&\frac{\zeta}{6}&\frac{1}{3}&\frac{-\zeta+1}{6}\\
\frac{-\zeta+1}{9}&\frac{-\zeta+2}{9}&\frac{-\zeta+3}{9}&\frac{\zeta+2}{9}&\frac{\zeta+1}{9}&\frac{\zeta}{9}&&\frac{\zeta}{4}&\frac{1}{4}&\frac{-\zeta+1}{4}&\frac{\zeta}{4}&\frac{1}{4}&\frac{-\zeta+1}{4}\\
\frac{-\zeta+2}{9}&\frac{\zeta+1}{9}&\frac{-\zeta+2}{9}&\frac{\zeta+1}{9}&\frac{-\zeta+2}{9}&\frac{\zeta+1}{9}&&\frac{\zeta+1}{12}&\frac{1}{4}&\frac{-\zeta+2}{12}&\frac{\zeta+1}{12}&\frac{1}{4}&\frac{-\zeta+2}{12}\\
\frac{1}{3}&\frac{\zeta-1}{3}&\frac{-2\zeta+2}{3}&\frac{2\zeta}{3}&\frac{-\zeta}{3}&\frac{1}{3}&&\frac{1}{12}&\frac{-\zeta+3}{12}&\frac{-\zeta+3}{12}&\frac{\zeta+2}{12}&\frac{\zeta+2}{12}&\frac{1}{12}\\
\frac{-\zeta+1}{12}&\frac{-2\zeta+3}{12}&\frac{-\zeta+4}{12}&\frac{\zeta+3}{12}&\frac{2\zeta+1}{12}&\frac{\zeta}{12}&&\frac{\zeta}{3}&\frac{-\zeta+1}{3}&\frac{\zeta}{3}&\frac{-\zeta+1}{3}&\frac{\zeta}{3}&\frac{-\zeta+1}{3}\\
\frac{-\zeta+2}{12}&\frac{1}{4}&\frac{\zeta+1}{12}&\frac{-\zeta+2}{12}&\frac{1}{4}&\frac{\zeta+1}{12}&&\frac{\zeta+1}{9}&\frac{-\zeta+2}{9}&\frac{\zeta+1}{9}&\frac{-\zeta+2}{9}&\frac{\zeta+1}{9}&\frac{-\zeta+2}{9}\\
\frac{-\zeta+2}{18}&\frac{-\zeta+2}{9}&\frac{-\zeta+5}{18}&\frac{\zeta+4}{18}&\frac{\zeta+1}{9}&\frac{\zeta+1}{18}&&\frac{\zeta+1}{6}&\frac{0}{1}&\frac{-\zeta+2}{6}&\frac{\zeta+1}{6}&\frac{0}{1}&\frac{-\zeta+2}{6}
\end{tabular}
}
\end{center}
\end{table}
We set up checks of
Table~\ref{twodicefiftyone}
in several tools but only one calculation completed. The open source package \texttt{Bertini}~\cite{Bertini} [which uses homtotopy continuation algorithms] impressed us by finding numerical approximations~\calcpagecite{2.bertini.htm} to all $51$ solutions.  

Table~\ref{twodicefiftyone} reproves Proposition \ref{crapsfairno} revealing a bit more geometry than the proof by elimination. It piqued our interest in what happens for dice with any number of sides, and suggested that the natural context for answering these questions needed to allow arbitrary complex ``probabilities''.
\fi

\section{The total-to-parts problem}\label{totalpartssetup}
After establishing some results for small examples computationally,  we formulated a much more general problem in the hopes of finding a unifying theory. We begin in \subsec{defnot} by making some preliminary definitions and establishing notation that will be used throughout the sequel. This leads to \subsec{totalparts} where we state the general problem and some attractive special cases and review what we now know about their answers. 
  In \subsec{geometric}, we describe a more geometric way of thinking about these problems that suggests further variations on these questions and, for the reader with no exposure to algebraic geometry, give some elementary examples that we hope make clear the notions that we draw on.

\subsection{Definitions and notation}\label{defnot}

We will index the sides of dice by $\sq{k} := \{0, 1, \ldots, k-1\}$ and \emph{sacks} or collections of dice by $\ang{n} := \{1, 2, \ldots, n\}$. Numbering sides from $0$ rather than $1$ makes possible uniform indexing and simpler formulae for total distributions, for example, in (\ref{ft}) and (\ref{multpolyformula}). 

By a \emph{strict die} $\dd = (d_0, d_1, \ldots, d_{k-1})$ of \emph{order} $k$, we mean a probability space whose underlying set is $\sq{k}$ (indexed throughout by $i$) with probability distribution given by $\PPP(i) = d_i$, subject, of course, to the standard conditions $d_i \in \RR$, $d_i \ge 0$ for $ i\in \sq{k}$, and $ \sum_{i\in \sq{k}}d_i =1$.
The \emph{fair} $k$-die has the uniform distribution  with all $d_i = \frac{1}{k}$; any other $k$-die is \emph{unfair} or \emph{loaded}. 
It will be convenient to consider  \emph{$k$-pseudodice} $\pp = (p_0, p_1, \ldots, p_{k-1})$ for which the $p_i$ are pseudoprobabilities allowed to take arbitrary \emph{complex} values and required to satisfy only $\sum_{i\in \sq{k}}p_i =1$. See Remark~\ref{sumone}, for an explanation of why we retain only the ``sum to $1$'' condition for a probability distribution.

We will continue to denote by $\dd$ a die known to be strict and by $\pp$ a general pseudodie, and when no confusion is likely, we omit the qualifiers ``strict'' and ``pseudo''.  A \emph{coin} is a die of order $2$ and a die is \emph{real} if all its probabilities are real numbers.	

We define a \emph{sack} $\SS:= (\pp^1, \pp^2, \dots, \pp^n)$ to be a finite ordered set of dice (indexed throughout by $j \in \ang{n}$). We call the list $\kk_S := (k_1, k_2, \ldots k_n)$ of orders of these dice the \emph{type} of $S$ and call the probability distributions of the individual dice in $\SS$ (in order) its \emph{part} probabilities. To the sack $\SS$, we associate the product sample space $K_\SS := \prod_{j \in \ang{n}} \sq{k_j}$ (indexed throughout by $\ii$) of order $k_\SS := \prod_{j \in \ang{n}} k_j$. If $\ii = (i_1, i_2, \ldots, i_n) \in K_\SS$, we define the \emph{total} $t_{\ii} := \sum_{j=1}^n i_j$. If we set $T_\SS := \sum_{j \in \ang{n}} (k_j-1)$, this allows us to define a partition of $K_\SS$ by the sets $K_t := \{\ii \in I~|~ t_{\ii} = t\}$ for $t =0, 1, \dots,  T_\SS$. When $\SS$ is understood, we often omit it as a subscript and, likewise, we often omit the modifiers total or part when no confusion can result.

We will always assume that the dice in $\SS$ are \emph{independent}. For strict sacks, this means that we have the usual product formula for the probability of the outcome $\ii$ when the dice in the sack $\SS$ are rolled, 
$\PPP(\ii) = \prod_{j \in \ang{n}} p_{i_j}^j$,
and we use this formula to define $\PPP(\ii)$ for any sack. 

We can then define the $\thst{t}{th}$-total probability $f_t$ of $\SS$ to be the (pseudo)probability of observing a total of $t$ when the dice in the sack $\SS$ are rolled. 
\begin{equation}\label{ft}
f_t = \sum_{\ii \in K_t} \PPP(\ii) = \sum_{\ii \in K_t} \Bigl(\prod_{j \in \ang{n}} p_{i_j}^j\Bigr)\,. 
\end{equation}
and define the \emph{total probability distribution} $\ff_{\SS} := (f_0, f_1, \cdots, f_T)$. 

We call this total distribution \emph{fair} if it equals the total distribution on the sack of fair dice having the same orders as those in $\SS$ and, if so, we call $\SS$ (totally) \emph{fair}. Such a sack is \emph{exotic} if at least one of its dice is unfair. It is easy to check that if a strict sack of dice is totally fair, then so is the subsack obtained by removing any fair dice. Hence, any fair sack of dice contains a distinguished exotic subsack consisting of its unfair dice.

The \emph{reverse} of a die $\dd$ is the die $\dd'$ whose probability vector is that of $\dd$ in reverse order. If we reverse all the dice in a sack, we get a reverse sack whose total probability vector is also reversed. A die is \emph{palindromic} if it equals its reverse, and a sack is palindromic if its total probability distribution is. A palindromic sack may contain non-palindromic dice: see Example~\ref{cointhree}.

\subsection{The total-to-parts question and special cases of it}\label{totalparts}

We can now state the most general problem we seek to investigate here, and strict and fair special cases, the last of which motivates this paper. As we will see in the sequel, the flavor of the fair version is arithmetic, of the strict version somewhat more probabilistic, and of the general version algebro-geometric. 

\begin{ttsque}[General]\label{generalquestion}
When does the total probability distribution on a sack of dice determine the part probabilities of the dice, up to permutation?
\end{ttsque}

\begin{ex}\label{generalpairofcoins} To get a feel for this question, let's work out the generalization of a pair of coins from Example~\ref{twocoinsfair}. The general total distribution is given by numbers $r$, $s$ and $t$ summing to $1$, and the probabilities $p$ and $q$  of heads on the two coins must then satisfy
\begin{displaymath}
\begin{array}{rcl}
pq & = & r\\
p(1-q)+(1-p)q & = & s\\
(1-p)(1-q) & = &  t\\
\end{array}
\end{displaymath}
From the first equation, we get $q=\frac{r}{p}$, and eliminating $q$ from  second gives $p^2-(2r+s)p-r=0$. Using $r+s+t=1$, this has discriminant $D=s^2-4rt$ so $p = \frac{(2r+s)\pm \sqrt{D}}{2}$. These roots are swapped when the coins are, so the answer to Question~\ref{generalquestion} for two coins is yes.
\end{ex}

It is not hard to show that the answer to Question~\ref{generalquestion} is, in general, no (see Example~\ref{cointhree}). In the special case of sacks of any number of coins treated in \sect{coins}, the answer is positive. We will also see in Corollary~\ref{generalpartstototal} that this holds in general, \emph{only} for sacks of coins (and, of course, singleton sacks).

\begin{ttsque}[Strict]\label{strictquestion}
Does the probability distribution on a strict sack of dice determine the probabilities of the dice, up to permutation?
\end{ttsque}

In general, the answer is no---again, see Example~\ref{cointhree}. Indeed, the answer is positive only when the answer to \question{generalquestion} is. Much more subtle, as we shall see, is the following question which turns out to have an arithmetic flavor.

\begin{ttsque}[Fair]\label{fairquestion}
If the total probability distribution on a strict sack of dice is fair, must all the dice be fair? Contrapositively, do there exist exotic sacks of dice?
\end{ttsque}

Specializing to a sack of two $6$ sided dice, we get the question in the title of the paper.  We confirm the negative answers of Example~\ref{twosixsideddice} and 
\ifshrink
\calcpagecite{2.solutions}
\else
Table~\ref{twodicefiftyone}
\fi
by a more elegant method in Corollary~\ref{crapsfairnobis}. We also show that there are no exotic sacks consisting of a coin and a die (see \subsec{coindiefair} and Remark~\ref{coindiebis}). 

Computer calculations reviewed in \sect{exoticsacks} show that totally fair sacks of two dice of certain small orders must be fair. In the other direction, for most pairs of orders, exotic sacks exist and the number of such sacks gets large as the orders do. Historically, the first we found was a pair of $13$-sided dice, but there are even simpler examples (see Table~\ref{decatab}).  These examples make it clear how to produce exotic sacks with more dice and lead to some intriguing asymptotic conjectures. 

\subsection{Geometric reformulations and variants}\label{geometric}

Fix a type $\kk= (k_1, k_2, \ldots, k_n)$ of sack and set $U := \sum_{j \in \ang{n}} (k_j)$ and $T := \sum_{j \in \ang{n}} (k_j-1) = U-n$. Then we can take $p^j_i$ for $ j \in \ang{n}$ and $ i \in \sq{k_j}$ as coordinates on $\CCC^U$, and use the $n$ independent equations ${\sum_{i\in \sq{k_j}}p^j_i =1}$ to define a linear subspace $\VVV \isom \CCC^T$. Equivalently, we can use the $p^j_i$ for $ i \in \sq{k_j-1}$ as coordinates on $\CCC^T$ and use these equations to implicitly define~$p^j_{k_j}$. 

Thus sacks $\SS:= (\pp^1, \pp^2, \dots, \pp^n)$ of type $\kk$ can be identified with points of  either $\VVV$ or of $\CCC^T$ [or, more algebraically, of the affine variety $\AAA^T$].  Equation~\eqref{ft} can then be viewed as defining a part-to-total mapping $\FF_{\kk}: \VVV \to \CCC^T$ or $\FF_{\kk}: \CCC^T \to \CCC^T$ by $\SS = (\pp^1, \pp^2, \dots, \pp^n)\,\to\, \ff_\SS$. Since the source and target are of the same dimension, we expect the map $\FF_\kk$ to be \emph{finite}. 

Finiteness is a geometer's term that means that for \emph{general} points $\ff$, the number of solutions of the equations $\FF_{\kk}(\SS) = \ff$---the order of the \emph{fiber} $\ff^{-1}(\{\ff\})$ of $\FF_{\kk}$ over $\ff$---has a common finite value called the \emph{degree} of $\FF_\kk$. Here ``general'' is a somewhat vague but very convenient notion, indicating that there is a non-zero polynomial vanishing on the complementary set of ``special'' $\ff$.  The special fibers may be infinite, or empty, or contain a smaller finite number of points [at each of which, typically, several points of nearby general fibers come together or \emph{ramify} (cf. Example~\ref{symex} and Remark~\ref{diceramification})]. Milne's lecture notes~\cite{Milne} provide an accessible account of finiteness (Chapter~8) and of prerequisite notions (Chapters~1-7).

\begin{ex}\label{symex} A model example is the map $\FF:\CCC^n \to \CCC^n$ sending $\pp := (p_1, \ldots p_m)$ to $\ee= (e_1, \ldots, e_n)$ where the $p_j$ are the roots \emph{in a fixed order} of a monic complex polynomial $P(x)$ of degree $n$ and the $e_j$ are the coefficients of $P(x)$ in degree order: 
$$P(x) = \prod_{j=1}^n (x-p_j) = x^n + \sum_{i=1}^{n} e_ix^{n-i}\,.$$ 
Here $e_i$ is the \thst{i}{th}-elementary symmetric function of the  $p_j$: that is, the sum of the $\binom{n}{i}$ squarefree monomials of degree $i$ in the $p_j$. The fundamental theorem of symmetric polynomials~\cite{Cox}*{Theorem~2.2.7} implies that the map $\FF$ is surjective of degree $n!$ and that the fiber containing $\pp$ is obtained by permuting the $p_j$. In particular, every fiber is finite (and non-empty). Thus, $\ee$ general means that $P(x)$ has distinct roots or equivalently that the degree $n$ discriminant~~\cite{Lang}*{pp.161--162}  is non-zero.
\end{ex}

\begin{ex}\label{finiteex} The map $f:\CCC^3 \to \CCC^3$ given by $f(x,y,z) = (a,b,c) := (yz,xz,xy)$ has degree $2$. When $xyz\not=0$, $y = \frac{c}{x}$ and $z= \frac{b}{x}$ so $a= yz =  \frac{bc}{x^2}$ and hence $x = \pm \sqrt{\frac{bc}{a}}$. Fixing the sign then determines $y$ and $z$ by symmetry. However, the fiber over $(0,0,0)$ is the union of the $x$, $y$ and $z$ axes. If $a$ is non-zero, then the fiber over $(a,0,0)$ is the hyperbola $yz=a$ in the plane $x=0$ and, if $b$ is also non-zero, then the fiber $(a, b, 0)$ is empty, since $c=0$ forces $x$ or $y$ (and hence $b$ or $a$) to be $0$ too.
\end{ex}

A weaker version of~\question{generalquestion} asks whether finiteness always holds here.

\begin{ttsque}[Finite map]\label{finitequestion}
Are the part-to-total maps always finite? If so, what is the degree $\deg(\FF_\kk)$ and what are the special fibers?
\end{ttsque}

A positive answer to this question allows, as in Example~\ref{finiteex}, for the collection of sacks $\SS'$ having the same total distribution as $\SS$ to be infinite, but asserts that for general $\SS$ this collection is finite. A stronger version, though one that unlike~\question{generalquestion} allows for $\SS'$ that are not reorderings of $\SS$, asks whether we have a picture like that in Example~\ref{symex} with \emph{all} fibers finite and non-empty.

\begin{ttsque}[Non-empty finite fibers]\label{strongfinitequestion}
Do the equations $\FF_\kk(\SS) = \ff$ always have only a non-zero, finite number of solutions?
\end{ttsque}

\begin{rem}\label{pairofcoinsrem} Example~\ref{generalpairofcoins} shows that for a pair of coins, the answer to these questions is positive: the degree is $2$ with the total distribution determining the pair of dice up to order, and over the discriminant locus defined by the vanishing of $s^2-4rt$ we have a single pair of identical dice with head probabilities $p = \frac{2r+s}{2}$ (cf. Example~\ref{twocoinsfair}).
\end{rem}
In Theorem~\ref{strongfinitethm}, we show that the answer to \question{strongfinitequestion} (and, \textit{a fortiori}, to \question{finitequestion}) is positive in general and, in Corollary~\ref{degpartstototal}, give a multinomial formula for  $\deg(\FF_\kk)$.

\section{Sacks of coins}\label{coins}
\stepcounter{subsection}% delete if subsections used

This section should, historically, fall between \subsec{totalparts} and \subsec{geometric}: after making the extensions in the former, we studied the simplest examples, sacks of $n$ coins, which we were able to completely resolve by straightforward combinatorial arguments, and it was this solution that led to the geometric reformulation of the latter. We moved this argument back because it builds on Examples~\ref{generalpairofcoins} and~\ref{symex}.

For a sack of $n$ coins, the sample space $K_\SS$ is simply the power set of $\ang{n}$, and the total $t$ of any roll (or toss) $\ii$ is simply the number of $j$ for which  the index $i_j$ equals $1$, so the set $K_t$ of rolls with total $t$ can be identified with subsets $J$ of $\ang{n}$ of order $t$. We can further simplify notation by writing $p_j$ for $p^j_0$ and replacing $p^j_1$ by $(1-p_j)$.  Making these substitutions, equation~\eqref{ft} then becomes
\[
f_t = \sum_{\substack{J\subset \ang{n}\\ |J|=t~}}~ \biggl[\Bigl(\prod_{j\in J} p_j\Bigr) \Bigl(\prod_{j\not\in J} (1- p_j)\Bigr)\biggr].
\]

\begin{prop} Let $\ee_{\SS} = (e_o,e_1, \ldots, e_n)$ where $e_i $ is the $\thst{i}{th}$ elementary function of the probabilities $p_j$ (cf. Example~\ref{symex}). Then:
	\begin{enumerate}
		\item For $ t \in \sq{t}$,  $f_t = {\ds{\sum_{k=t}^n}} (-1)^{k-t}~{\binom{k}{t}}~ e_k$.
		\item For $j \in \ang{n}$, $\sspan \{ f_n, f_{n-1}, \ldots,  f_{j}\} = \sspan \{ e_n, e_{n-1}, \ldots,  e_{j}\}$.
	\end{enumerate}
\end{prop}
\begin{proof} The first claim follows directly from the definition of $f_t$ as follows. Fix any of the $\binom{n}{t}$ terms $P$ in the sum defining $f_t$. The second product in $P$ is a binomial of total degree $(n-t)$ containing $\binom{n-t}{k-t}$ ``interior'' terms of $p$-degree $(k-t)$ and sign $(-1)^{k-t}$. Multiplying these degree $(k-t)$ ``interior'' terms by the first product in $P$ and summing over $P$ gives a squarefree degree $k$ homogeneous symmetric polynomial in the $p_j$, which must therefore be of the form $c e_k$. Since $e_k$ contains $\binom{n}{k}$ terms, we find that 
\[
c = (-1)^{k-t}\left(\frac{\binom{n}{t}\cdot \binom{n-t}{k-t}}{\binom{n}{k}}\right)\,.
\]
A straightforward calculation shows that this simplifies to $c=(-1)^{k-t}\binom{k}{t}$ but there is also a nice bijective argument for this. The denominator counts choices of a $k$ element subset and the numerator choices of such a subset in $2$ stages, by first choosing $t$ of the elements and then the other $(k-t)$. In this second count, each underlying $k$-set arises once for each of its $t$ element subsets.   

The first claim shows that $\ff_{\SS} = U\ee_{\SS}$ where $U$ is an upper-triangular matrix that has ones on the diagonal. As $U$ is invertible, the second claim is immediate.
\end{proof}

The second claim of the proposition shows that $\ff_{\SS}$ determines $\ee_{\SS}$. But the elementary symmetric polynomials generate the ring of all symmetric polynomials (cf. Examples~\ref{symex}), so $\ee_{\SS}$ determines the $p_j$, up to permutation, and hence $\SS$ up to permutation of the dice. 

\begin{cor}\label{generalcoins} The total distribution of any sack of coins determines the part distributions of the coins, up to permutation. In other words, when all $k_j$ equal $2$, the answer to~\question{generalquestion} is positive.
\end{cor}

\section{A coin and a die}\label{coindie}

For a sack consisting of a coin and a die, we can eliminate the die probabilities to obtain an equation for the coin probability.

\subsection {The fair case}\label{coindiefair}

\begin{lem}\label{coindielem} If a coin and a die have fair total distribution, then both are fair.
\end{lem}
\begin{proof} Writing the pseudoprobabilites of the coin as $(p,1-p)$ and of the die as $(q_0, q_1, \ldots q_{k-1})$, we can express the fairness of the total distribution by the equations 
\begin{center}
\begin{tabular}{>{$}l<{$}>{$}c<{$}>{$}l<{$}>{$}c<{$}>{$}c<{$}}
	pq_0 & & &= &\frac{1}{2k}\\
	pq_1 & + & (1-p)q_0&=  &\frac{1}{k}\\
	\kern6pt\vdots &  & \kern18pt\vdots &&\vdots\\
	pq_{k-1} & + & (1-p)q_{k-2}&= & \frac{1}{k}\\
	&  & (1-p)q_{k-1}&= & \frac{1}{2k}
\end{tabular}
\end{center}
Note that the first and last equations imply that $p$ and $1-p$ are both non-zero, which will allow us to use the equations in succession to solve for and eliminate the quantities $kq_j$. We start with $kq_0= \frac{1}{2p}$, then rewrite successive equations as $kq_j = \frac{1}{p} + \frac{p-1}{p}q_{j-1}$ to obtain, inductively, 
\[
k q_j = \frac{1}{p}+ \frac{p-1}{p^2}+ \cdots+ \frac{(p-1)^{j-1}}{p^{j}} + \frac{1}{2}\frac{(p-1)^{j}}{p^{j+1}}\,,
\]
and finally use the last equation in the form $kq_{k-1}= -\frac{1}{2(p-1)}$ to see, after clearing denominators, that $p$ satisfies the equation
\[
\frac{1}{2}p^{k}+ (p-1)p^{k-1} + \cdots + (p-1)^{k-1}p + \frac{1}{2}(p-1)^{k}\,.
\]

If we split each interior term in two equal parts $\frac{1}{2} (p-1)^jp^{k-j}$ and group one set of these halves with the first term and the other with the last, we obtain two $k$-term geometric progressions with common ratio $r=\frac{p-1}{p}$ and respective initial terms $a'=\frac{1}{2}p^k$ and $a''=\frac{1}{2}(p-1)p^{k-1}$. This equals a single progression with ratio $r$ and initial term $a= a'+a''= (p-\frac{1}{2})p^{k-1}$. Such a progression can sum to $0$ only if either $a=0$ or $r=-1$ (and $k$ is even). In either case, we must have $p =\frac{1}{2}$, so the coin is fair, and then the total distribution equations inductively yield $q_j=\frac{1}{k}$ for all $j$, so the die is fair too.
\end{proof}

\subsection {The general case}\label{coindiegeneral}

A version of the same analysis can be carried out for a coin and a die with a general total distribution to obtain an equation for $p$ in terms of the vector $(f_0, f_1, \ldots, f_k)$ of total probabilities. It shows that 
\[
p^k + \sum_{i=0}^{k-1} a_ip^i = 0 \text{\quad where \quad} 
a_i = \sum_{j=0}^{i} (-1)^{k-i} \binom{k-j}{k-i} f_j\,.
\]

We will leave the details to the interested reader because it turns out that an unfair total distribution no longer determines the parts, even in the strict case. The simplest examples---there are many, as we will see in \sect{exoticsacks}---involve a coin and a three sided die.

\begin{ex}\label{cointhree} The total distribution $(\frac{1}{9}, \frac{7}{18}, \frac{7}{18}, \frac{1}{9})$ is common to the three sacks with part distributions $(\frac{1}{2}, \frac{1}{2})$ and $(\frac{2}{9}, \frac{5}{9}, \frac{2}{9})$,  $(\frac{1}{3}, \frac{2}{3})$ and $(\frac{1}{3}, \frac{1}{2}, \frac{1}{6})$, and the reverse of the second.
\end{ex}

\section{The simplifying viewpoint}\label{multpoly}

\subsection{Stillman's observation and some immediate consequences} What Mike Stillman said about the part-to-total map $\FF: \CCC^T \to \CCC^T$ for sacks of dice of a given type $\kk$ that dropped our jaws was, ``That's polynomial multiplication''. He was seeing the problem in terms of generating functions, nicely introduced in~\cite{Concrete}*{Chapter~7}, which provide a powerful tool for assembling and relating combinatorial data. Here we associate to any pseudodie $\pp$ of order $k$ the \emph{distribution polynomial} $\pp(x) = \sum_{i=0}^{k-1} p_ix^i$. Conversely, every polynomial of degree $(k-1)$ or less (less, because we allow  $p_i=0$) with coefficients summing to $1$ is associated to a unique die of order $k$. Likewise, the total distribution of any sack $\SS$ yields a polynomial $\ff_{\SS}(x) = \sum_{i=0}^{T} f_ix^i$ of degree at most~$T$. 

A moment's inspection of the total distribution equation~(\ref{ft}) shows that it can now be restated much more concisely as:
\begin{equation}\label{multpolyformula}
	\ff_{\SS}(x) = \prod_{\pp\in \SS} \pp(x)
\end{equation}
This was Stillman's insight. From it, we immediately obtain positive answers to \question{finitequestion} and \question{strongfinitequestion}.
\begin{thm}\label{strongfinitethm} For dice of a fixed type $\kk$, all fibers of  the part-to-total map $\FF_{\kk}$ are finite.
\end{thm}
\begin{proof} If $\SS$ is any sack of type $\kk$, then over the complex numbers, which are algebraically closed, $\ff_\SS(x)$ factors completely into (at most) $T$ monic linear factors times a non-zero scalar~\cite{Hungerford}*{Theorem~3.19}. What Stillman's observation means is that a sack $\SS'$ has total distribution $\ff_{\SS}(x)$ if and only if the multiplicity of every root $z$ of $\ff_{\SS}(x)$---that is, the number of linear factors $(x-z)$---equals the sum over $\pp \in \SS'$ of the multiplicities of $z$ in $\pp(x)$. The roots only determine each polynomial up to a non-zero homothety, but this is fixed in each case by the pseudoprobability condition that the coefficients sum to $1$.  In other words, such $\SS'$ correspond to partitions of the linear factors of $\ff_{\SS}(x)$ into subsets of sizes at most $k_j$. But there are only finitely many such partitions. 
\end{proof}

With a bit more work, we can give a formula for the degree of $\FF_{\kk}$. For general sacks of type $\kk$ all the pseudoprobabilities $p_i^j$ will be non-zero---that is, this condition fails only on the lower dimensional set where at least one of the equations  $p_i^j=0$ holds, so $\pp_j(x)$ will have degree exactly $k_j$. Likewise, for general sacks $\SS$, all the complex roots of $\ff_{\SS}(x)$ will be distinct: the polynomial that vanishes if this is not true is the degree $T$ discriminant~\cite{Lang}*{pp.161--162}. So the $\SS'$ in the same fiber as $\SS$ will correspond bijectively to partitions of the set of $T$ roots into exactly $n$ parts of sizes $k_j-1$. Choosing these parts in succession introduces no ambiguity because, although we have not required that all the $k_j$ be distinct, we work with \emph{ordered sacks}. Thus 
\begin{cor}\label{degpartstototal}
The degree of the map $\FF_{\kk}$ is $\displaystyle{\frac{T!}{\prod_{j=1}^n (k_j-1)!}}$.
\end{cor}

\begin{exs}For pairs of coins this degree is $\frac{2!}{(1!)^2} = 2$. The associated total polynomial, $tx^2+sx+r$, will be general and have two preimages when it has distinct roots---i.e. when the discriminant $s^2-4rt$ does not vanish (cf. Examples~\ref{generalpairofcoins} and Remark~\ref{pairofcoinsrem}). If $k_1=2$ and $k_2=3$ so $T=3$ and the three roots are distinct, then the degree is $\frac{3!}{2!1!} = 3$. Taking the roots to be $-1$, $-2$ and $-\frac{1}{2}$ gives Example~\ref{cointhree} in which the three cases correspond to which root is chosen for the coin. As the example makes clear, this kind of ambiguity is typical, with or without a restriction to the strict case. For pairs of $6$-sided dice, the degree is $\frac{10!}{(5!)^2} = 252$ confirming \texttt{Magma}'s calculation~\calcpagecite{2.magma-b.htm}. 
\end{exs}

\begin{cor}\label{generalpartstototal}
The total distribution determines the parts up to permutation exactly for sacks of coins and singleton sacks.
\end{cor}
\begin{proof}
For a sack of $n$ coins, all $k_j=2$ and $T = n$ so the degree of $\FF_{\kk}$ is $n!$: this reproves Corollary~\ref{generalcoins}. In general, if we define $m_p$ to be the number of $j$ for which $k_j=p$ then $m_{\kk}:= \prod_p (m_p!)$ gives the number of permutations (by exchanging parts of the same order) that preserve a set of distribution polynomials that is general in the sense discussed above. The number $m_{\kk}$ always divides the degree of $\FF_{\kk}$ (because the quotient counts so-called distinguishable partitions) but the quotient is not $1$ unless all $k_j$ are equal to $2$ or there is only one die because, when we transpose a pair of roots from different factors, we only transpose the factors themselves when they are both of degree one. 
\end{proof}

Over the locus of sacks where $\ff_{\SS}(x)$ has repeated roots---in particular, when different dice in the sack have roots in common---the fibers have smaller order. Counting these fibers [and saying what sheets come together in each] is a more intricate problem, but the arguments above make it clear that this problem is fundamentally combinatorial. We will not enter into it here, except for pairs of totally fair dice in the next subsection (cf. Lemma~\ref{faircount} and Remark~\ref{diceramification}).

Instead we now turn to the fair case, which motivated this paper and which is more subtle. A fair die $\dd$ of order $k$ has polynomial $\frac{1}{k}\psi_k(x)$ where  $\psi_k(x):= (x^{k-1} + x^{k-2} + \cdots + x + 1) = \frac{(x^k-1)}{(x-1)}$. The associated roots are the roots of unity $\zeta_{m,k} := e^{2\pi i\frac{m}{k}}$  with $-\frac{k}{2} < m \le \frac{k}{2}$ of orders dividing $k$ and not equal to $1$~\cite{Lang}*{p.116}. Of these, only $-1= \zeta_{m,2m}$ (for even $k$), corresponds to a real factor $(x+1)$. Note that, by Euler's Identity,
$\zeta_{m,k}= \cos\bigl(2\pi\frac{m}{k}\bigr) + i\sin\bigl(2\pi\frac{m}{k}\bigr)$, which yields:

\begin{lem}\label{fairfactors} The roots $\zeta_{m,k}$ and $\zeta_{-m,k}$ are complex conjugates with real part $\cos(2\pi\frac{m}{k})$ and norm $1$, and hence the monic irreducible factors over $\RR$ of $\psi_k(x)$ are  $(x+1)$ for all even $k$ and $\chi_{m,k}(x) = (x^2 - 2\cos(2\pi\frac{m}{k}) + 1)$ for $1 \le m < \frac{k}{2}$. 
\end{lem}

\begin{rem}\label{coindiebis}
	If a totally fair sack contains a coin and a die of order $k$, no transposition of roots is possible if $k$ is odd, and only the exchange of the factors $(x+1)$ is possible if $k$ is even. We immediately recover Lemma~\ref{coindielem}. 
\end{rem}

\subsection{Pairs of totally fair dice} In this subsection, we analyze pairs $\dd$ and $\ddhat$ of totally fair dice of order $k$, with total polynomial $\Psi(x) =\frac{1}{k^2}\psi_k(x)^2 $ in more detail. 

We  first count such sacks, beginning with the observation that $\dd(x)$ (and hence $\ddhat(x)$) is determined by the vector $\rr := (r_1, \ldots r_{k-1})$ giving the multiplicities of the roots $\zeta_{m,k}$ which must satisfy the multiplicity inequalities
$ 0 \le r_m \le 2, m \in \ang{k-1}$ and the degree equality $\sum_{m=1}^{k-1} r_m = k-1$. If exactly $\ell$ of the $r_m$ equal $2$, then by the equality exactly $k-1-2\ell$ of the remaining $(k-1-\ell)$ will equal $1$ and the other $\ell$ will equal $0$. The number of such $\rr$ is simply the number of choices for the subsets of $\ang{k-1}$ with $r_m=2$ and $r_m=1$, respectively $\binom{k-1}{\ell}$ and $\binom{k-\ell-1}{k-1-2\ell}$.

\begin{lem}\label{faircount} The number of pairs of totally fair dice of order $k$ is $$\ds{\sum_{\ell=0}^{\lfloor\frac{k-1}{2}\rfloor}\binom{k-1}{\ell}\binom{k-\ell-1}{k-1-2\ell}}\,.$$
\end{lem}

For example, when $k=2$, there is a unique such sack, and when $k=6$, there are $\binom{5}{0}\binom{5}{5}+ \binom{5}{1}\binom{4}{3}+\binom{5}{2}\binom{3}{1}= 1+20+30=51$, confirming the count of solutions in
\ifshrink
\calcpagecite{2.solutions}.
\else
Table~\ref{twodicefiftyone}\footnote{In fact, the source for the typeset table was produced by using SAGE~\cite{SAGE} to list all pairs $(\dd, \ddhat)$ by the procedure leading to the count above, and passing this output to \texttt{Maple}\texttrademark~\cite{Maple} to have the fractions arranged in a more compact form and then add the wrappers needed to format the table.}.
\fi

\begin{rem} \label{diceramification}
A quick check on all the counts above can be obtained by working out how many of the points of a general fiber of the part-to-total map for a pair of dice of order $k$, for which the $2(k-1)$ roots of $\ff_{\SS}(x)$ are distinct, come together at each point of the totally fair fiber, where the roots are a multiset with $(k-1)$ distinct elements each occurring twice. To say that two $(k-1)$-element subsets of the latter multiset are equal as multisets (and hence give equal polynomials) means exactly that the corresponding vectors $\rr$ of multiplicities are equal. When $r_m$ is $0$ or $2$, there is no ambiguity about which subset of the \thst{m}{th} pair is being chosen but when $r_m=1$ there are $2$ subsets. Thus, for a point indexed in the $\thst{\ell}{th}$ term of the sum in Lemma~\ref{faircount} where $(k-1-2\ell)$ of the $r_m$ equal $1$, there will be $2^{(k-1-2\ell)}$ points of the general fiber coming together. For example, when $k=6$ and the general fiber has $252$ elements, we obtain the check $2^5\cdot 1+2^3\cdot 20+2^1\cdot 30 = 252$.
\end{rem}

\begin{prop}\label{ddfactors} Suppose that $\dd$ and $\ddhat$ are a pair of strict dice of order $k$ whose  total distribution is fair.	
	\begin{enumerate}
		\item If $k$ is odd, $\dd(x) = c \prod_m \bigl(\chi_{m,k}(x)\bigr)^{r_m}$ and $\ddhat(x) = \chat \prod_m \bigl(\chi_{m,k}(x)\bigr)^{(2-{r_m})}$ where $0\le r_m \le 2$, $\sum_m r_m = \frac{k-1}{2}$ and $c$ and $\chat$ are non-zero scalars. 
		\item If $k$ is even, we get the same conclusion except that $\sum_m r_m = \frac{k-2}{2}$ and each of $\dd(x)$ and $\ddhat(x)$ must also contain one of the factors $(x+1)$ in $\psi_k$.
	\end{enumerate}	
\end{prop}
\begin{proof} For $k$ odd, this is immediate from the hypothesis $\dd\cdot\ddhat = (\psi_k)^2$, Lemma~\ref{fairfactors}, and the fact that $\dd$ and $\ddhat$ are both real. For $k$ even, the only additional observation is that parity forces each to have exactly one of the two factors $x+1$ in the right hand side. 
\end{proof}

As a first application, we give a very short third proof of Proposition~\ref{crapsfairno}.

\begin{cor}\label{crapsfairnobis} You can't play a fair game of craps with a loaded pair of dice.
\end{cor}
\begin{proof}
Since $\psi_k = \chi_{1,6}(x)\chi_{2,6}(x) (x+1)$, the only sack, other than a pair of fair dice, allowed by Proposition~\ref{ddfactors} has $\dd(x)$ and $\ddhat(x)$ multiples of  
\[
\bigl(\chi_{1,6}(x)\bigr)^2 (x+1)= (x^2-x+1)^2 (x+1) =x^5 - x^4 + x^3 + x^2 - x + 1
\]
and 
\[\bigl(\chi_{2,6}(x)\bigr)^2 (x+1) = (x^2+x+1)^2 (x+1)= x^5 + 3x^4 + 5x^3 + 5x^2 + 3x + 1\,.
\]
This sack is not strict: it yields the second row of
\ifshrink
\calcpagecite{2.solutions}.
\else
Table~\ref{twodicefiftyone}.
\fi
\end{proof}

\begin{rem} \label{sumone}
Theorem~\ref{strongfinitethm} does not hold without some non-degeneracy condition on our polynomials because if any $\pp(x)=0$ so is $\ff(x)$. The normalization that the coefficients sums to $1$ fits our probability context much more naturally than, say, requiring the $\pp(x)$ to be monic. 
\end{rem}

\section{A menagerie of exotic pairs}\label{exoticsacks}

The calculations reported through \sect{coindie} (and many we have not included) led us to suspect that the answer to~\question{fairquestion} was positive. Stillman's insight quickly led us to counterexamples.

\subsection{Existence of exotic pairs}\label{exoticpairs}

Let's start, as we did in the event, with pairs of dice of the same order.  By now our strategy is clear: pick an order $k$, list all the sacks $\dd(x)$ and $\ddhat(x)$ permitted by Proposition~\ref{ddfactors}, and look for strict sacks. \textit{A priori}, we can think of no reason why such sacks should exist, but \textit{a posteriori}, we learn that they almost always do.  Here is the first $13$-sided example we found.

\begin{ex}\label{tridecadahedralex} For simplicity, we suppress the index $13$ and the dependence on $x$. Consider the dice
$\dd = c\,\,\chi_1\cdot\chi_2\cdot\chi_3\cdot\chi_4\cdot\chi_4\cdot\chi_6$
and
$\ddhat = \chat\,\, \chi_1\cdot\chi_2\cdot\chi_3\cdot\chi_5\cdot\chi_5\cdot\chi_6$.
in which we have swapped the $\chi_4$ and $\chi_5$ factors of a fair pair of dice. By Lemma~\ref{fairfactors}, the total distribution of this sack is fair. Both dice are palindromic and the lower ``half'' of each probability vector is shown, approximately, in Table ~\ref{tridecatab}.	
{\renewcommand{\arraystretch}{1.1}\renewcommand{\tabcolsep}{3pt}
\begin{center}	\refstepcounter{equation}\label{tridecatab}\vskip3pt\centerline{\textbf{Table~\ref{tridecatab}} Numerical probabilities for tridecahedral dice of Example~\ref{tridecadahedralex} }
\small\begin{tabular}{c>{$}c<{$}>{$}c<{$}>{$}c<{$}>{$}c<{$}>{$}c<{$}>{$}c<{$}>{$}c<{$}}
& \text{\normalsize$d_0$} & \text{\normalsize$d_1$} &\text{\normalsize$d_2$} &\text{\normalsize$d_3$} &\text{\normalsize$d_4$} &\text{\normalsize$d_5$} &\text{\normalsize$d_6$} \\
	{\normalsize $\dd$}& 0.0992916&0.0210685&0.1381701&0.0410895&0.0693196&0.1241391&0.0138431\\
	{\normalsize $\ddhat$}& 0.0595938&0.1065425&0.0732460&0.0499115&0.0997570&0.0877406&0.0464172\\
	\end{tabular}\vskip3pt
\end{center}
}
The coefficients in Table~\ref{tridecatab} are roundings to $7$ places of coefficients computed from $14$ place values for the cosines that occur in the $\chi_m$, so the positivity of these coefficients---that is, the strictness of the sack---is also unimpeachable. 
\end{ex}

The \emph{smallest} exotic pair is not Example~\ref{tridecadahedralex} but is obtained from a fair pair of dice of order $10$ by swapping the $\chi_3$ and $\chi_4$ factors.  As the upper half of Table~\ref{decatab} shows, this example is strict but seems a bit of a cheat because four of the $\dd$-probabilities are zero.  There are $3$ exotic pairs of order $12$. Two obtained by swapping $\chi_2$ and $\chi_3$ or $\chi_3$ and $\chi_4$ in a fair pair of dice  have rational part probabilities, but again some are zero, as the reader may check.  The smallest exotic pair with all probabilities positive, shown in the lower half of  Table~\ref{decatab}, is obtained by swapping $\chi_4$ and $\chi_5$. 

\begin{center}
\refstepcounter{equation}\label{decatab}\vskip3pt\centerline{\textbf{Table~\ref{decatab}} Probabilities of exotic pairs of dice with $10$ and $12$ sides}\nopagebreak
{\small
\begin{tabular}{>{$}c<{$}rr>{$}c<{$}>{$}c<{$}>{$}c<{$}>{$}c<{$}>{$}c<{$}>{$}c<{$}}
\text{Order} & Swap & Face & 1 & 2 & 3 & 4 & 5  & 6  \\
\multirow{2}{*}{$10$} & \multirow{2}{*}{$\chi_3,\chi_4$} & {\normalsize $\dd_{10}$}  & {\frac{5-\sqrt{5}}{20}} & 0 & {\frac{\sqrt{5}}{10}} & 0 & {\frac{5-\sqrt{5}}{20}}&  \\
& & {\normalsize $\ddhat_{10}$} & {\frac{5+\sqrt{5}}{100}} & {\frac{5+\sqrt{5}}{50}} &  {\frac{1}{10}} & {\frac{5-\sqrt{5}}{50}}&  {\frac{15-\sqrt{5}}{100}} &  \\[8pt]
\multirow{2}{*}{$12$}&\multirow{2}{*}{$\chi_4,\chi_5$} & {\normalsize $\dd_{12}$}  & {\frac{2-\sqrt{3}}{4}} & {\frac{2\sqrt{3}-3}{4}}  & {\frac{2\sqrt{3}-3}{4}} & \frac{2-\sqrt{3}}{4} & {\frac{2\sqrt{3}-3}{4}}& \frac{2-\sqrt{3}}{4} \\
&& {\normalsize $\ddhat_{12}$} & {\frac{2+\sqrt{3}}{36}} & {\frac{1}{36}} &  {\frac{4+\sqrt{3}}{36}} & {\frac{2-\sqrt{3}}{36}}&  {\frac{5}{36}} & {\frac{4-\sqrt{3}}{36}}
\end{tabular}\vskip3pt
}%small
\end{center}

We leave the reader to find the smallest exotic sack, of type $(3,4)$. For many others, found using \texttt{Macaulay2}, \texttt{Magma}, and \texttt{Maple}\texttrademark~\cites{Macaulay, Magma, Maple} see \calctoccite{7.1}.

\subsection{Asymptotics of exotic pairs}\label{asymexoticpairs}

Similar calculations yield lots of exotic sacks of many combinatorial types. Here we will consider only sacks of $2$ dice, where we have evidence~\calctoccite{7.2} for the existence of such sacks for almost all pairs of orders, and further, for asymptotic predictions about the numbers of such sacks. We summarize our evidence in two Conjecture-Problems that we challenge interested readers to take up\footnote{If you do, please communicate your results so we can update the webpage~\calcpagecite{progress.htm} which we invite you to visit to learn about (and to avoid duplicating) work of others. We do not ourselves plan any further work, except as mentors to interested students.}.   First, our calculations (and Lemma~\ref{coindielem} for the case $k=2$) confirm the list of exceptions below and suggest its completeness.

\begin{conjprob} \label{exoticpairsexist}
There are exotic sacks of $2$ dice of every pair of orders $(k,k')$ with $2 \le k \le k'$ except for $(2, k')$, $(3,3)$, $(3,6)$, $(3,9)$, $(4,4)$, $(4,8)$, $(5,5)$, $(6,6)$, $(7,7)$, $(8,8)$, $(9,9)$, and $(11,11)$.
\end{conjprob}

Inspection of our data suggests that the growth of the number of such sacks with the orders of the dice exhibits asymptotic regularities that call for explanation. We content ourselves with giving precise conjectures when one die has order~$3$.

\begin{conjprob} \label{exotictriangles} If $S_3(k)$ be the set $m$ with $1 \le m <\frac{k}{2}$ such that the sack obtained by exchanging the $\chi_m$ factor of a fair $k$-die with the $\frac{1}{3}(x^2+x+1)$ of a fair $3$-die is strict (hence exotic unless $k = 3m$), $M_3(k) := \max S_3(k)$ and $R_3(k) := \frac{M_3(k)}{k}$, then
\begin{enumerate}
\item Any $m$ between $ \lceil\frac{k}{4}\rceil$ and $ M_3(k)$ is in $S_3(k)$.  
\item For $k \ge 336$, $M_3(k) \ge \lfloor\frac{5k}{12}\rfloor$.
\item $M_3(k) \le \frac{60}{143}k$ with equality exactly when $k$ is a multiple of $143$.
\item \label{exoticitem} $M_3(k+143) - M_3(k)=60$, except that there is a sequence $b_a$ with $b_1 = 0$, $b_{a+1}-b_a$ either $0$ or $1$, and such that, if $a$ is not divisible by $143$ and $k = 603a+143b_a$, then $M_3(k+143) - M_3(k)=59$.
\item The $\limsup$ as $k \to \infty$ of $R_3(k)$ equals $\frac{60}{143} \simeq 0.4195804$.
\end{enumerate} 
\end{conjprob}

A few remarks are in order. Since $\frac{k}{4} \le m$ just ensures that the new $3$-die is strict, the point of these conjectures is to pin down $M_3(k)$ and $R_3(k)$. As Figure~\ref{triscatter} shows, their behavior warns us to exercise caution when making conjectures from inductive evidence. For example, $\frac{M_3(k)}{k} = \frac{5}{12}$ exactly when $k=12\ell$ and \emph{when $\ell \le 28$}! For some time we were sure, based on computations to $k=500$, that $M_3(k+143) = M_3(k)+60$ always held, and the surprising appearance of exceptions involving $603$ makes us a little nervous that others, possible with much larger modulus, may be lurking. The sequence $b_a$ and the $\liminf$ of $R_3(k)$  are mysterious. Empirically, $b_a$ takes on each value either $4$ for $5$ times, which leads to a lower bound for the $R_3(k)$ about $10^{-5}$ less than $\frac{60}{143}$ but we do not see enough of a pattern to conjecture a value for the $\liminf$. A SAGE~\cite{SAGE} notebook~\calcpagecite{7.SAGE-a.htm} is listing the exceptions up to $k=10^6$. See~\calcpagecite{progress.htm} for those up to $k=10^5$ or for word of counterexamples to Conjecture~\ref{exotictriangles}.(\ref{exoticitem}) found after going to press.
\begin{center}
    \tikzstyle{background grid}=[draw, black!50,step=.5cm]
	\begin{tikzpicture}[]{
		\node [inner sep=0pt,above right] (image) at (0,0)
            {\includegraphics[scale=0.60]{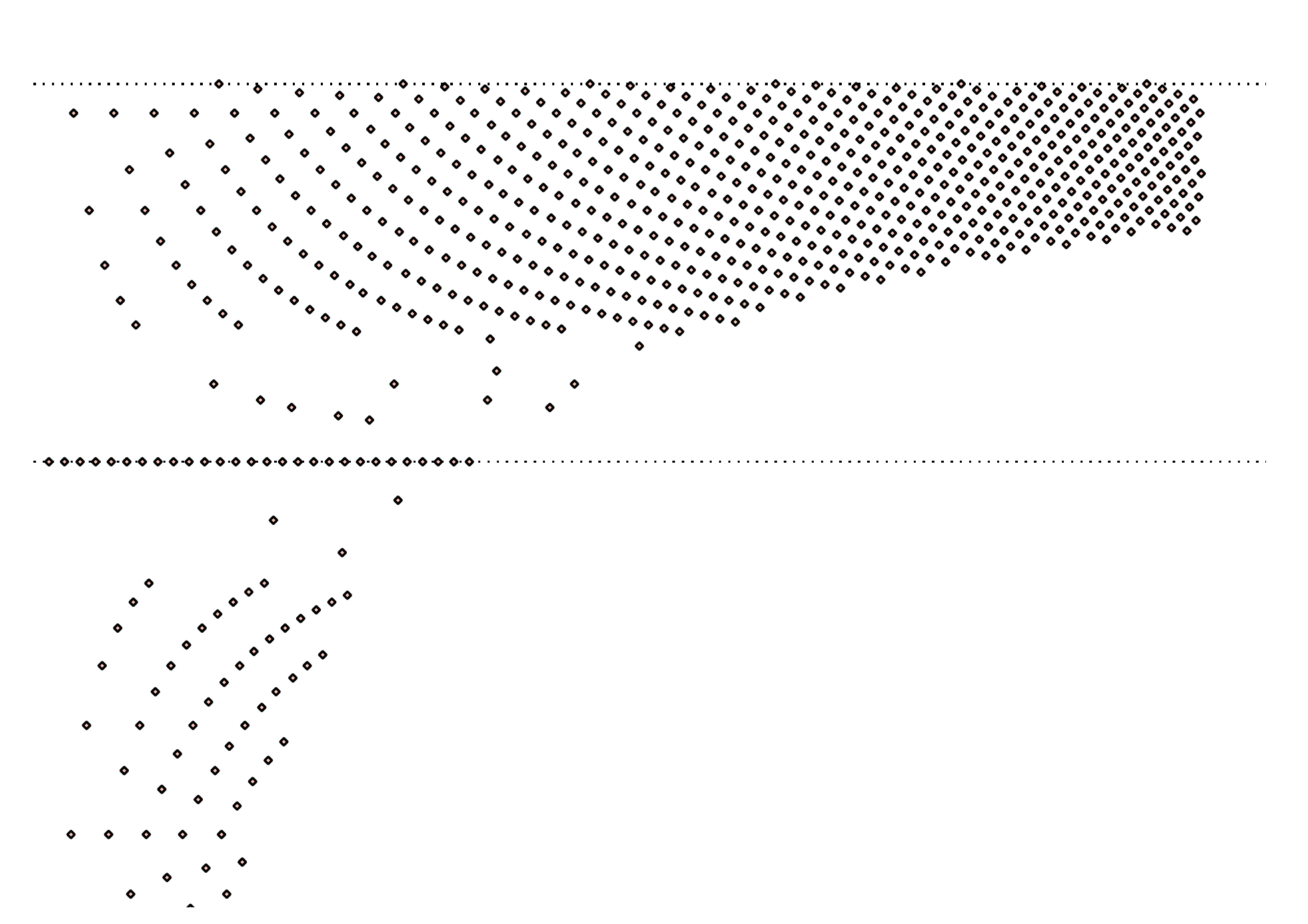}};
 		\begin{scope}[x={(image.south east)},y={(image.north west)}]        
			\fill (.360,.544) circle (2pt);
			\fill (.167,.990) circle (2pt);
			\fill (.310,.990) circle (2pt);
			\fill (.452,.990) circle (2pt);
			\fill (.594,.990) circle (2pt);
			\fill (.737,.990) circle (2pt);
			\fill (.879,.990) circle (2pt);
			\fill (.766,.783) circle (2pt);
			\draw (.963,.540) node [] {$\frac{5}{12}$};
			\draw (.963,.990) node [] {$\frac{60}{143}$};
			\draw (.64,.28) node [] {
			\begin{minipage}{6cm} \small
			The circular markers at the top show the points with $k$ a multiple of $143$. The other two markers show the point with $k=336$ with smaller multiples of $12$ directly to its left, and with $\text{$k=746=603+143$}$, the first of the exceptions in Conjecture-Problem~\ref{exotictriangles}(\ref{exoticitem}). A few points with $k <20$ lie below the plot.
			\end{minipage}};
		\end{scope}
		}
    \end{tikzpicture}
\refstepcounter{equation}\label{triscatter}\centerline{\textbf{Figure~\ref{triscatter}} Scatter plot of $(k, \frac{M_3(k)}{k})$ for $k$ up to $950$}
\end{center}

What happens if, instead of a $(3,k)$-sacks, we try to enumerate $(\ell,k)$-sacks for other fixed values of $\ell$? For small $\ell$, analogues of Conjecture~\ref{exotictriangles} hold and these can, at least at first, be much simpler (if less intriguing). For example, our calculations suggest that $S_4(k)$ is the interval extending from $\lceil\frac{k}{6}\rceil$ to $\lfloor\frac{k}{3}\rfloor$. 

But, complications soon arise as it becomes necessary to consider exchanges of sets of several factors $\chi_m$.
\ifshrink
\else
The swaps (each given by its $m$ and ordered first by the number swapped and then lexicographically) for $k=20$ are $[3\leftrightarrow 4]$, $[4\leftrightarrow 5]$, $[5\leftrightarrow 6]$, $[6\leftrightarrow 7]$, $[6\leftrightarrow 8]$, $[7\leftrightarrow 8]$, $[8\leftrightarrow 9]$, $[3,7 \leftrightarrow 4,6]$, $[3,7 \leftrightarrow 4,8]$,  $[4,9 \leftrightarrow 5,8]$, $[5,9\leftrightarrow 6,8]$, $[6,8\leftrightarrow 7,9]$, and for $k=21$ are $[3\leftrightarrow 4]$, $[4\leftrightarrow 5]$, $[5\leftrightarrow 6]$, $[6\leftrightarrow 7]$, $[7\leftrightarrow 8]$, $[8\leftrightarrow 9]$, $[9\leftrightarrow 10]$, $[2,5\leftrightarrow 3,6]$,$[3,7\leftrightarrow 4,8]$, $[4,10\leftrightarrow 5,9]$,  $[4,10\leftrightarrow 6,9]$, $[5,8\leftrightarrow 6,9]$, $[5,9\leftrightarrow 6,8]$, $[5,10\leftrightarrow 6,9]$, $[6,10\leftrightarrow 7,9]$, $[7,10\leftrightarrow 8,9]$, $[3,8,9\leftrightarrow 4,7,10]$and $[4,8,9\leftrightarrow 5,7,10]$. Roughly speaking, you can swap two sets of factors when the averages of the $m$ in each are sufficiently close. 

\fi
This suggests that the growth of $E(k)$ is likely to be exponential in $k$ but the computations we have made are too limited to provide convincing qualitative evidence, let alone to suggest precise conjectures. Table~\ref{ektab} shows the first few diagonal counts of exotic pairs of $k$-dice which already show super-linear growth. We leave these questions to the interested reader.
\begin{center}\renewcommand{\arraystretch}{1.}
\refstepcounter{equation}\label{ektab}\vskip3pt\centerline{\textbf{Table~\ref{ektab}} Number $E(k)$ of exotic pairs for small values of the order $k$ }
\begin{tabular}{r>{$}c<{$}>{$}c<{$}>{$}c<{$}>{$}c<{$}>{$}c<{$}>{$}c<{$}>{$}c<{$}>{$}c<{$}>{$}c<{$}>{$}c<{$}>{$}c<{$}>{$}c<{$}>{$}c<{$}>{$}c<{$}}
$k$ & 12 & 13 & 14 & 15 & 16 & 17 & 18 & 19 & 20 & 21 & 22 & 23 & 24 & 25  \\
$E(k)$ &  3& 2& 3& 4 & 4 & 6 & 7 & 8 & 12 & 18 & 19 & 27 & 42 & 60 
\end{tabular}\vskip3pt
\end{center}

\ifshrink\else
\section{Appendix: The game of craps}\label{craps}
\stepcounter{subsection}% delete if subsections used

We first recall here how the game of craps is played, ignoring secondary aspects that affect only betting on the game and not its outcome. Then we explain how the probability of winning it may be calculated, not because this is essential to the main theme of the paper, but because the arguments provide a very pretty application of basic ideas of finite probability to an infinite sample space and are increasingly rarely covered in contemporary probability courses. 

The game is played with two standard cubical dice with faces numbered from $1$ to $6$%
\ifshrink%
\else%
\footnote{So all rolls of single die differ by $1$ from our convention in subsection~\ref{defnot}, and $2$-dice totals differ by~$2$.}
\fi%
. There are two stages, in both of which, the outcome of play is determined by the total of the numbers showing on the two dice. In the first stage, consisting of single \emph{comeout} roll $t$, the player wins by throwing a $7$ or $11$ and loses by throwing a $2$, $3$ or $12$. Rolls of $4$--$6$ and $8$--$10$ lead to a second stage in which this first roll becomes the players \emph{point}. In this second stage, the player wins by re-rolling the comeout point $t$, loses by rolling a $7$ and rolls again in all other cases. So the second stage contains outcomes involving any finite number of rolls. 

{\renewcommand{\arraystretch}{1.4}
\begin{center}
	\refstepcounter{equation}\label{crapstable} \textbf{Table~\ref{crapstable}} Probabilities arising in finding the chance of winning at craps.
\begin{tabular}{>{$}c<{$}>{$}c<{$}>{$}c<{$}>{$}c<{$}>{$}c<{$}>{$}c<{$}>{$}c<{$}>{$}c<{$}>{$}c<{$}>{$}c<{$}>{$}c<{$}>{$}c<{$}}
\text{Total $t$ on first roll} & 2 & 3 & 4 & 5 & 6 & 7 & 8 & 9 & 10 & 11 & 12	\\
\PPP(t) & \frac{1}{36}& \frac{2}{36}& \frac{3}{36}& \frac{4}{36}& \frac{5}{36}& \frac{6}{36}& \frac{5}{36}& \frac{4}{36}& \frac{3}{36}& \frac{2}{36}& \frac{1}{36} \\
\PPP(w | t) & 0& 0& \frac{3}{9}& \frac{4}{10}& \frac{5}{11}& 1& \frac{5}{11}& \frac{4}{10}& \frac{3}{9}& 1& 0 \\
\PPP(t \cap w) & 0& 0& \frac{9}{324}& \frac{16}{360}& \frac{25}{396}& 1& \frac{25}{396}& \frac{16}{360}& \frac{9}{324}& 1& 0\\
\end{tabular}
\end{center}
\smallskip
}%arraystretch

Table~\ref{crapstable} summarizes the ingredients that go into finding the probability of a win for the player, an event that we denote $w$. Since the different totals are mutually exclusive, the probability that the player wins at craps can be computed from the table as 
\[
\PPP(w) = \sum_{t=2}^{12} \PPP(t \cap w) = \frac{243}{495} \simeq  0.4909\,.
\]

How can we check the values in the table? Those in the first row are standard counts. Given the total $t$, the number $i$ on the first die determines that on the second and $i$ must be between $1$ and $t-1$ if $t \le 7$ and between $t-6$ and $6$ if $t \ge 7$. Given the numbers in the second row, those in the third follow by applying the Intersection Formula for probabilities, $\PPP(t \cap w) = \PPP(t)\cdot \PPP(w | t)$. 

The conditional probabilities in the second row are more interesting. They can be computed from a tree diagram, but this diagram has the feature, seldom found in the examples treated in finite probability courses, of being infinite. An initial segment of the branch of this tree, starting from the node on the left where a comeout point of $t=9$ has just been rolled, is shown in Figure~\ref{crapsfigure}. The tree branches up to a winning leaf (shaded black) on a roll of $9$, down to a losing leaf (shaded gray) on a roll of $7$ and across on any other role. Each edge is labeled with the probability of following it from its left endpoint, and each leaf is labeled with the probability of reaching it from the root of the tree. 
\begin{center}
	\begin{tikzpicture}[scale=0.85]
		\foreach \x in {0,1,...,3} {%
			\coordinate (left) at (3*\x,0);
			\coordinate (top) at (3*\x+2,2);
			\coordinate (right) at (3*\x+3,0);
			\coordinate (bot) at (3*\x+2,-2);
			\draw (left) -- (top)  node [midway, fill=white] {$\frac{4}{36}$} node [above right] {$\frac{4}{36}\cdot \left(\frac{26}{36}\right)^{\x}$};
			\filldraw [black] (top) circle (3pt);
			\draw (left) -- (right) node [midway, fill=white] {$\frac{26}{36}$};
			\draw (left) -- (bot) node [midway, fill=white] {$\frac{6}{36}$} node [below right] {$\frac{6}{36}\cdot \left(\frac{26}{36}\right)^{\x}$};
			\filldraw [gray] (bot) circle (3pt);
			}
		\draw[dashed] (12,0) -- (13.5,0);	
	\end{tikzpicture}
\vspace{6pt}
\refstepcounter{equation}\label{crapsfigure} \textbf{Figure~\ref{crapsfigure}} Tree diagram for the game of craps \emph{after} a comeout roll of $9$.
\end{center}
The tree makes visible a formula for $\PPP(w | 9)$ as the sum $\frac{a}{1-r}=\frac{4}{10}$ of the geometric series with initial term $ a= \frac{4}{36}$ and ratio $r=\frac{26}{36}$. 

But there is a much easier way to see obtain this value from the tree, by noticing that we move to a leaf on any outcome in the event $ (7 \text{~or~} 9)$ and win when this outcome is a $9$. Hence $\PPP(w | 9) = \PPP\bigl(9 |\, (7 \text{~or~} 9)\bigr) = \frac{\PPP(9)}{\PPP(7 \text{~or~} 9)}$ which the first row of Table~\ref{crapstable} gives immediately as $\frac{4}{10}$. The other non-trivial entries in the second row of the table then follow analogously from the first row. 

\fi

\bibspread

\newcounter{lastbib}
\setcounter{lastbib}{\value{bib}}

\section*{Software packages referenced}

\setcounter{lastbib}{\value{bib}}

\section*{Computer-assisted calculations referenced}

\end{document}